\documentclass[12pt,reqno,draft]{article} 
\usepackage{amsmath,amssymb,amsthm,amsfonts, indentfirst, empheq}
\usepackage{enumerate,color,bm,graphicx,here}
\usepackage{multicol}
\makeatletter
\def\@cite#1#2{[{{\bfseries #1}\if@tempswa , #2\fi}]}
\renewcommand{\section}{%
\@startsection{section}{1}{\z@}
{0.5truecm plus -1ex minus -.2ex}%
{1.0ex plus .2ex}{\bfseries\large}}
\def\@seccntformat#1{\csname the#1\endcsname.\ }
\makeatother

\numberwithin{equation}{section} 
\theoremstyle{theorem}
\newtheorem{thm}{Theorem}[section]

\newtheorem{lem}[thm]{Lemma}

\theoremstyle{definition}
\newtheorem{df}{Definition}[section]
\newtheorem{remark}{Remark}[section]
\newtheorem*{prth1.1}{Proof of Theorem 1.1}
\newtheorem*{prth1.2}{Proof of Theorem 1.2}

\newcommand{\ep}{\varepsilon}
\newcommand{\pa}{\partial}
\newcommand{\R}{\mathbb{R}}
\newcommand{\N}{\mathbb{N}}
\newcommand{\cl}[1]{{\overline#1}}
\newcommand{\Ombar}{\cl{\Omega}}

\newcommand{\Tmaxe}{T_{{\rm max}, \ep}}

\newcommand{\ue}{u_{\ep}}

\newcommand{\ve}{v_{\ep}}

\newcommand{\io}{\int_{\Omega}}

\newcommand{\upto}{\nearrow}
\newcommand{\dwto}{\searrow}

\newcommand{\nrm}[2]{\| #1 \|_{ #2 (\Omega)}}

\newcommand{\norm}[2]{\| #1(\cdot, t) \|_{ #2 (\Omega)}}

\newcommand{\intd}[1]{{\rm d}#1}
\newcommand{\wcont}{C^0_{\rm w-\star}}
\newcommand{\wc}{\rightharpoonup}
\newcommand{\wsc}{\stackrel{\star}{\rightharpoonup}}

\newcommand{\baru}{\overline{u_0}}
\newcommand{\ub}{\overline{u}}
\newcommand{\vb}{\overline{v}}

\newcommand{\uue}{U_{\ep}}
\newcommand{\vve}{V_{\ep}}

\DeclareBoldMathCommand{\bmass}{1}
\begin{document}
\footnote[0]
    {2020{\it Mathematics Subject Classification}\/. 
    Primary: 35B35; Secondary: 35K55, 35Q92, 92C17.
    }
\footnote[0]
    {{\it Key words and phrases}\/: 
    Keller--Segel system; chemotaxis; flux limitation; nonlinear production.}
\begin{center}
    \Large{{\bf Stability of constant equilibria in a Keller--Segel system with gradient dependent chemotactic sensitivity and sublinear signal production}}
\end{center}
\vspace{5pt}
\begin{center}
    Shohei Kohatsu
   \footnote[0]{
    E-mail: 
    {\tt sho.kht.jp@gmail.com}
    }\\
    \vspace{12pt}
    Department of Mathematics, 
    Tokyo University of Science\\
    1-3, Kagurazaka, Shinjuku-ku, 
    Tokyo 162-8601, Japan\\
    \vspace{2pt}
\end{center}
\begin{center}    
    \small \today
\end{center}

\vspace{2pt}
\newenvironment{summary}
{\vspace{.5\baselineskip}\begin{list}{}{%
     \setlength{\baselineskip}{0.85\baselineskip}
     \setlength{\topsep}{0pt}
     \setlength{\leftmargin}{12mm}
     \setlength{\rightmargin}{12mm}
     \setlength{\listparindent}{0mm}
     \setlength{\itemindent}{\listparindent}
     \setlength{\parsep}{0pt}
     \item\relax}}{\end{list}\vspace{.5\baselineskip}}
\begin{summary}
{\footnotesize {\bf Abstract.}
This paper deals with the homogeneous Neumann boundary-value problem
for the Keller--Segel system
\[
   \begin{cases}
    u_t=\Delta u - \chi \nabla \cdot (u|\nabla v|^{p-2}\nabla v),
  \\
    v_t=\Delta v - v + u^{\theta}
   \end{cases}
\]
in $n$-dimensional bounded smooth domains
for suitably regular nonnegative initial data,
where $\chi > 0$, $p \in (1, \infty)$ and $\theta \in (0,1]$.
Under smallness conditions on $p$ and $\theta$,
we prove that the spatially homogeneous equilibrium solution
is stable.
This generalizes the result in \cite{KoY1}
from the case $\theta = 1$ to more general values of $\theta$.
} 
\end{summary}
\vspace{10pt}

\newpage
\section{Introduction}\label{Sec:Intro}

This paper deals with the Keller--Segel system with gradient dependent
chemotactic sensitivity and sublinear signal production,
\begin{align}\label{KSPP}
  \begin{cases}
    u_t=\Delta u - \chi \nabla \cdot (u |\nabla v|^{p-2} \nabla v)
    & \mbox{in}\ \Omega \times (0, \infty),
  \\
    v_t = \Delta v - v + u^{\theta}
    & \mbox{in}\ \Omega \times (0, \infty),
  \\
    \nabla u \cdot \nu=\nabla v \cdot \nu=0
    & \mbox{on}\ \pa\Omega \times (0, \infty),
  \\
    u(\cdot ,0) = u_0, \ v(\cdot, 0) = v_0
    & \mbox{in}\ \Omega
  \end{cases}
\end{align}
in a smoothly bounded domain $\Omega \subset \R^n$ $(n \in \N)$,
where $\chi > 0$, $p \in (1, \infty)$ and $\theta \in (0,1]$ are constants,
where $\nu$ is the outward normal vector to $\pa\Omega$,
and where $u_0, v_0$ satisfy
%
\begin{align}\label{initial}
\begin{cases}
		u_0 \in C^0(\cl{\Omega}), 
	\quad 
		u_0 \ge 0\ \,{\rm in}\ \cl{\Omega}
	\quad 
		{\rm and}
	\quad 
		u_0 \neq 0,
	\\
		v_0 \in W^{1,\infty}(\Omega)
	\quad
		{\rm and}
	\quad
		v_0 \ge 0\ \,{\rm in}\ \cl{\Omega}.
\end{cases}
\end{align}

When $p = 2$ and $\theta =1$, the system \eqref{KSPP} is reduced to the classical Keller--Segel
system proposed in \cite{KS-1970, KS-1971}
to model the behavior of slime molds with density $u$,
which are attracted by the chemoattractant with density $v$ they produce themselves.
In this framework, global existence and large-time behavior of solutions have been investigated
(see for instance \cite{OY-2001, W-2010, C-2015}, and also \cite{W-2014, C-2017} for
the case of logistic source).
The case of nonlinear production $(\mbox{i.e.}\ \theta \neq 1)$ was initially studied by
Liu and Tao~\cite{LT-2016},
where they obtained global existence and boundedness of classical solutions,
provided that $0 < \theta < \frac{2}{n}$.
On the other hand,
Winkler~\cite{W-2018} considered the parabolic--elliptic setting,
and proved existence of initial data such that the corresponding solutions
to the system blow up in finite time
when $\theta > \frac{2}{n}$.
Similar results for
the other chemotaxis systems with nonlinear production can be found
in \cite{GST-2016, HT-2017, VW-2020, FV-2021}, for example.

In light of refined approaches in the recent modeling literature
(e.g.\ \cite{BBNS-2010}),
chemotaxis systems with the flux-limited term $-\chi \nabla \cdot (u|\nabla v|^{p-2}\nabla v)$
have been introduced by Negreanu and Tello~\cite{NT-2018},
where they considered the parabolic--elliptic relevant of \eqref{KSPP} with $\theta = 1$,
and obtained uniform boundedness of $u(\cdot, t)$ in $L^{\infty}(\Omega)$ when
$p \in (1, \infty)$ $(n=1)$, and $p \in (1, \frac{n}{n-1})$ $(n \ge 2)$.
For the parabolic--parabolic case, Yan and Li~\cite{YL-2020}
studied the system \eqref{KSPP} with $\theta = 1$,
and established global existence and boundedness of weak solutions
in the case $p \in (1, \frac{n}{n-1})$ $(n \ge 2)$.
More recently, in \cite{KoY1},
stability of constant equilibria was established in the case that
$p \in (1,2)$ $(n = 1)$,
and that $p \in (1, \frac{n}{n-1})$ $(n \ge 2)$,
and also its asymptotic stability was revealed when $p \in [2, \infty)$ and $n = 1$.
The case of chemotaxis-fluid system with flux limitation and nonlinear production
can be found in \cite{W-2023}.
However,
large-time behavior of solutions to flux-limited model with nonlinear production,
or even with sublinear production (i.e.\ $\theta < 1$), has not
been investigated yet.
Thus the main purpose of this paper is to reveal behavior of solutions to
the system \eqref{KSPP} with $\theta < 1$ and $p \in (1, \infty)$.

Due to singularity of the flux-limited term $-\chi \nabla \cdot (u|\nabla v|^{p-2}\nabla v)$,
we cannot expect global existence of \emph{classical solutions} to \eqref{KSPP}
in general.
Thus, we will consider existence of \emph{weak solutions} to \eqref{KSPP} in the following
sense.

%
%
\begin{df}\label{df1}
Let $u_0$ and $v_0$ satisfy \eqref{initial}.
Let $T > 0$.
A pair $(u,v)$ of functions is called a \emph{weak solution} of \eqref{KSPP} in $\Omega \times (0,T)$ if
\begin{enumerate}[{\rm (i)}]
\item
$u \in \wcont ([0, T); L^{\infty}(\Omega))$, and
$v \in L^{\infty}( \cl{\Omega} \times [0,T) ) \cap L^2( [0,T); W^{1,2}(\Omega) )$,
\item
$u \ge 0 \quad \mbox{a.e.}\ \mbox{on}\ \Omega \times (0,T)$, and
$v \ge 0 \quad \mbox{a.e.}\ \mbox{on}\ \Omega \times (0,T)$,
\item
$| \nabla v |^{p-2} \nabla v \in L^2( \cl{\Omega} \times [0,T) )$,
\item $u$ has the mass conservation property
\[
    \int_{\Omega} u(\cdot, t) = \int_{\Omega} u_0 \quad\mbox{for a.a.}\ t \in (0, T),
\]
%
\item for any nonnegative $\varphi \in C_{\rm c}^{\infty}( \cl{\Omega} \times [0,T) )$,
%
\begin{align*}
    &- \int_{0}^{T} \int_{\Omega} u \varphi_t
    -\int_{\Omega} u_0 \varphi(\cdot, 0) 
    = \int_{0}^{T} \int_{\Omega} u \cdot \Delta \varphi
        + \chi \int_{0}^{T} \int_{\Omega} u | \nabla v |^{p-2} \nabla v \cdot \nabla \varphi,\\
    &-\int_{0}^{T} \int_{\Omega} v \varphi_t
        - \int_{\Omega} v_0 \varphi(\cdot, 0)
    = -\int_{0}^{T} \int_{\Omega} \nabla v \cdot \nabla \varphi
        - \int_{0}^{T} \int_{\Omega} v \varphi
        + \int_{0}^{T} \int_{\Omega} u^{\theta} \varphi.
\end{align*}
%
\end{enumerate}
%
If $(u,v) : \Omega \times (0, \infty) \to \R^2$ is a weak solution of \eqref{KSPP} in $\Omega \times (0,T)$ for all $T > 0$,
then $(u, v)$ is called a \emph{global weak solution} of \eqref{KSPP}.
\end{df}

Now the main result reads as follows.

%
%
\begin{thm}\label{thm1}
Let $n \in \N$.
Assume that $u_0$ and $v_0$ satisfy \eqref{initial}.
Let $\baru := \frac{1}{|\Omega|} \int_{\Omega} u_0$ and
$m := \| u_0 \|_{L^1(\Omega)}$.
Suppose that $\theta \in (0,1]$, and that
\begin{align}\label{con:p-weak}
  \begin{cases}
        p \in (1, \infty) 
        & \mbox{if}\ 0 < \theta \le \frac{1}{n},
\\ 
        p \in \left(1, \dfrac{n\theta}{n\theta-1}\right)
        &\mbox{if}\ \frac{1}{n} < \theta \le 1.
  \end{cases}
\end{align}
Then there exist a global weak solution $(u, v)$ of \eqref{KSPP}
%
and $t_1 > 0$ such that
\begin{align}\label{baru-m}
        \| u(\cdot, t) - \baru \|_{L^{\infty}(\Omega)}
        \leq C m^{\theta(p-1)+1} (1 + m^{\alpha} + m^{\beta})
    \quad
        \mbox{for all}\ t \ge t_1,
\end{align}
where $C > 0$, $\alpha > 0$ and $\beta > 0$ are constants.
In particular,
one can find $\eta_0 > 0$ such that for all $\eta \in (0, \eta_0)$,
whenever $u_0$ fulfills
\[
    \| u_0 \|_{L^1(\Omega)} \le \eta,
\]
%
$u$ satisfies
\begin{align}\label{stab_ex}
    \| u(\cdot, t) - \baru \|_{L^{\infty}(\Omega)} \le \eta
    \quad
        \mbox{for all}\ t \ge t_1.
\end{align}
\end{thm}
%
%
\begin{remark}
The terms $m^{\alpha}$ and $m^{\beta}$ in \eqref{baru-m} appear
as a result of estimating the integrals come from the flux-limited term by the
mass $m$
(Lemma~\ref{Lem:ueLr}).
Here we have to leave open the question, whether,
the values of $\alpha$ and $\beta$ in \eqref{baru-m} is optimal.
\end{remark}
%
%
\begin{remark}\label{u0min}
Considering $u - M$ and $v - M^{\theta}$
instead of $u$ and $v$
in Theorem~\ref{thm1},
where $M := \min_{x \in \cl{\Omega}} u_0 (x)$,
we can see that the property \eqref{stab_ex} can be established
even when $u_0$ fulfills
\begin{align}\label{stab}
  \| u_0 - M \|_{L^1(\Omega)} \le \eta
\end{align}
which means that the variation of $u_0$
in $\cl{\Omega}$ is sufficiently small.
We note that this implies stability of $\baru$.
Indeed, since
\begin{align}\label{indeed}
    \| u_0 - M \|_{L^1(\Omega)} \le |\Omega| \nrm{u_0 - \baru}{L^{\infty}}
\end{align}
holds, for any $\eta \in (0, \eta_0)$
we obtain \eqref{stab} from \eqref{indeed},
if we assume $\nrm{u_0 - \baru}{L^{\infty}} \le \delta := \frac{\eta}{|\Omega|}$.
\end{remark}

The proof of the main theorem is based on the arguments in \cite{KoY1}.
We will consider an approximate problem, and reveal its global classical solvability
in Section~\ref{Sec:App}.
In Lemma~\ref{Lem:uebound} we particularly derive uniform boundedness of
approximate solutions.
In Section~\ref{Sec:thm1} we establish \eqref{baru-m}
and complete the proof of Theorem~\ref{thm1}.
In order to obtain the estimate \eqref{baru-m},
we first prove that
  \[
   \| \ue(\cdot, t) - \baru \|_{L^{\infty}(\Omega)}
        \leq C m^{\theta(p-1)+1} (1 + m^{\alpha} + m^{\beta})
    \quad
        \mbox{for all}\ t \ge t_1
  \]
holds
with some $t_1 > 0$ which is independent of the approximate parameter $\ep$
(Lemma~\ref{Linf-uue}),
where $\ue$ is the first component of solutions to an approximate problem.
Then we let $\ep \dwto 0$ and construct a global weak solution
which satisfies \eqref{baru-m}.
The restriction $\theta \le 1$ will appear in Lemma~\ref{Lem:vexLq},
which gives an $L^q$-estimate for $\nabla \ve$, where $\ve$ is the second
component of solutions to an approximate problem.
The case $\theta > 1$ is left as an open.
%
%
\begin{remark}
Following the arguments in \cite[Lemmas 4.1--4.4]{JRT-2023},
we can construct a global weak solution $(u,v)$ of \eqref{KSPP}
which solves the second equation in the classical sense.
In this case, if we further assume that
$\nrm{v_0 - M^{\theta}}{L^1} \le \eta$ in Theorem~\ref{thm1}, then
we can conclude from \eqref{stab_ex}
with the aid of the parabolic regularity theory that
$\nrm{v(\cdot, t) - \baru^{\theta}}{L^{\infty}} \le \eta$
for all $t \ge t_1$,
which implies stability of constant equilibria $(\baru, \baru^{\theta})$.
\end{remark}

Throughout this paper,
we put $m := \|u_0\|_{L^1(\Omega)}$ as well as
$\baru := \frac{1}{|\Omega|}\io u_0$,
and we denote by $c_i$ generic positive constants.

\section{Global existence of approximate solutions}\label{Sec:App}

We will start by considering the approximate problem
\begin{align}\label{KSPP-R}
  \begin{cases}
    (\ue)_t=\Delta \ue - \chi \nabla \cdot 
                 \left(\ue (|\nabla \ve|^2 + \ep)^{\frac{p-2}{2}} \nabla \ve\right)
    & \mbox{in}\ \Omega \times (0, \infty),
  \\
    (\ve)_t=\Delta \ve - \ve + \ue^{\theta}
    & \mbox{in}\ \Omega \times (0, \infty),
  \\
    \nabla \ue \cdot \nu=\nabla \ve \cdot \nu=0
    & \mbox{on}\ \pa\Omega \times (0, \infty),
  \\
    \ue(\cdot,0)=u_0(x),\ \ve(\cdot,0)=v_0
    & \mbox{in}\ \Omega,
  \end{cases}
\end{align}
where $\ep \in (0,1)$.
Let us first give the result which states local existence in \eqref{KSPP-R}.

%
%
\begin{lem}\label{Lem:LocalApp}
Let $\chi > 0$, $p \in (1, \infty)$ and $\theta > 0$.
Then for each $\ep \in (0,1)$ there exist $\Tmaxe \in (0, \infty]$ and
uniquely determined functions
\begin{align}\label{Reg:Class}
      \begin{split}
        &\ue \in C^0(\Ombar \times [0, \Tmaxe))
        \cap C^{2,1}(\Ombar \times (0, \Tmaxe))
          \quad\mbox{and}
\\ 
        &\ve \in C^0(\Ombar \times [0, \Tmaxe))
        \cap C^{2,1}(\Ombar \times (0, \Tmaxe)) \cap
        \bigcap_{q > n} L^{\infty}_{\rm loc}
        ([0, \Tmaxe); W^{1,q}(\Omega))
  \end{split}
\end{align}
such that $(\ue, \ve)$ solves \eqref{KSPP-R}
in the classical sense in $\Omega \times (0, \Tmaxe)$,
that $\ue \ge 0$ and $\ve \ge 0$ in $\Ombar \times (0, \Tmaxe)$,
that fulfills
\begin{align}\label{LocalApp:mass}
    \int_{\Omega} \ue(\cdot, t) = \int_{\Omega} u_0
    = m
    \quad\mbox{for all}\ t \in (0, \Tmaxe),
\end{align}
and that
\begin{align}\label{Reg:criterion}
      \mbox{if}\ \Tmaxe < \infty,
  \quad\mbox{then}\quad
  \limsup_{t \upto \Tmaxe} \norm{\ue}{L^{\infty}} = \infty.
\end{align}
\end{lem}
\begin{proof}
%
Fixing $\ep \in (0,1)$,
we let
$X_{q,T} := C^0(\Ombar\times [0, T])\cap L^{\infty}(0,T; W^{1,q}(\Omega))$
for $q > n$ as well as $T>0$,
and consider the closed subset
\[
S_{q,T} := \{(\ue, \ve)\in X_{q,T} \mid \|(\ue, \ve)\|_{X_{q,T}} \le
R:=\|u_0\|_{L^{\infty}(\Omega)} + \|v_0\|_{W^{1,q}(\Omega)} + 1\}.
\]
Similar to the arguments in \cite[Lemma 2.2]{YL-2020},
for sufficiently small
$T=T(R)>0$ the map $\Phi = (\Phi_1, \Phi_2)$ on $S_{q,T}$, defined by
\begin{align*}
    &\Phi_1(\ue, \ve)(t) := e^{t\Delta}u_0 - \int_0^t e^{(t-s)\Delta}\nabla\cdot
      (\ue(|\nabla\ve|^2+\ep)^{\frac{p-2}{2}}\nabla\ve)(\cdot,s)\,\intd{s}
\\ \intertext{and}
    &\Phi_2(\ue, \ve)(t) := e^{t(\Delta - 1)} v_0 + \int_0^t e^{(t-s)(\Delta-1)} \ue^{\theta}(\cdot, s)\, \intd{s}
\end{align*}
for $(\ue,\ve)\in S_{q,T}$ and $t \in [0, T]$, is a contraction.
Thus, the Banach fixed point theorem enables us to find $(\ue, \ve) \in S_{q,T}$ which satisfies
$(\ue, \ve) = \Phi(\ue, \ve)$.
The criterion \eqref{Reg:criterion},
existence as a classical solution with the regularity \eqref{Reg:Class},
nonnegativity and uniqueness can be obtained in the same way as in
\cite[Lemma 3.1]{HW-2005}
(see also \cite[Lemma A.1]{TW-2014}),
while the property \eqref{LocalApp:mass} follows by integrating
the first equation of \eqref{KSPP-R}.
\end{proof}

In the sequel, for each $\ep \in (0,1)$ we let $\Tmaxe$ and $(\ue, \ve)$ be
as accordingly provided by Lemma~\ref{Lem:LocalApp}.
As a first step toward establishing global existence in \eqref{KSPP-R},
we derive an $L^q$-estimate for $\nabla \ve$ with some $q \ge 1$.

%
%
\begin{lem}\label{Lem:vexLq}
Let $\theta \in (0,1]$.
Then for any $q$ satisfying
\begin{align}\label{vexLq:q}
  \begin{cases}
        q \in [1, \infty] 
        & \mbox{if}\ 0 < \theta < \frac{1}{n},
\\ 
        q \in [1, \infty)
        &\mbox{if}\ \theta = \frac{1}{n},
\\ 
        q \in \left[1, \dfrac{n}{n\theta - 1}\right)
        &\mbox{if}\ \frac{1}{n} < \theta \le 1,
  \end{cases}
\end{align}
there exists $C > 0$ such that
\[
    \norm{\nabla \ve}{L^q} \le Ce^{-(1+\lambda_1)t} + Cm^{\theta}
    \quad\mbox{for all}\ t \in (0, \Tmaxe)\ \mbox{and}\ \ep \in (0,1).
\]
\end{lem}
\begin{proof}
As $\ve$ solves the second equation of \eqref{KSPP-R},
a combination of standard smoothing estimates for the Neumann
heat semigroup
(cf.\ \cite[Lemma~1.3~(ii) and (iii)]{W-2010})
and \eqref{LocalApp:mass}
yields
\begin{align*}
            \| \nabla \ve(\cdot, t) \|_{L^q(\Omega)}
        &\le e^{-t}
        \| \nabla e^{t \Delta} v_0 \|_{L^q(\Omega)}    
        + \int_{0}^{t} e^{-(t - \sigma)} 
            \|
                \nabla e^{ (t - \sigma) \Delta } \ue^{\theta}(\cdot, \sigma)
            \|_{L^q(\Omega)}
        \,\intd{\sigma}
\\
    &\le c_1 e^{-(1+\lambda_1)t} \nrm{\nabla v_0}{L^{\infty}}
\\
    &\quad\,+ c_1 \int_0^t (1 + (t - \sigma)^{-\frac{1}{2} - \frac{n}{2}(\theta - \frac{1}{q})})
        e^{-(1+\lambda_1)(t-\sigma)}\nrm{\ue(\cdot, \sigma)}{L^1}^{\theta}
        \,\intd{\sigma}
\\
    &\le c_1 e^{-(1+\lambda_1)t} \nrm{\nabla v_0}{L^{\infty}}
    + c_1 m^{\theta}
    \int_0^{\infty} (1 + \sigma^{-\frac{1}{2} - \frac{n}{2}
    (\theta - \frac{1}{q})}) e^{-(1+ \lambda_1)\sigma}\,\intd{\sigma}
\end{align*}
for all $t \in (0, \Tmaxe)$ and $\ep \in (0,1)$,
where $\lambda_1 > 0$ is the first nonzero
eigenvalue of $-\Delta$ in $\Omega$ under the Neumann boundary condition.
According to \eqref{vexLq:q}, we conclude
\[
c_2 := \int_0^{\infty} (1 + \sigma^{-\frac{1}{2} - \frac{n}{2}
    (\theta - \frac{1}{q})}) e^{-(1+ \lambda_1)\sigma}\,\intd{\sigma} < \infty,
\]
so that the statement follows for
$C := c_1 \max\{\nrm{\nabla v_0}{L^{\infty}}, c_2\}$.
\end{proof}

As a consequence of Lemma~\ref{Lem:vexLq},
we shall show that $\Tmaxe = \infty$ for every $\ep \in (0,1)$.

%
%
\begin{lem}\label{Lem:uebound}
Let $\theta \in (0,1]$, and suppose that $p$ satisfies \eqref{con:p-weak}.
Then there exists $C>0$ such that
\begin{align}\label{uebound:ue}
    \norm{\ue}{L^{\infty}} \le C
    \quad
    \mbox{for all}\ t \in (0, \Tmaxe)\ \mbox{and}\ \ep \in (0,1).
\end{align}
In particular, we have
\begin{align}\label{uebound:Tmaxe}
    \Tmaxe = \infty
    \quad\mbox{for any}\ \ep \in (0,1).
\end{align}
\end{lem}
\begin{proof}
From the representation formula for $\ue$ we have
\begin{align}
\notag
    &\norm{\ue}{L^{\infty}}
\\
   &\quad\, \le \nrm{e^{t\Delta} u_0}{L^{\infty}}
    + \chi \int_0^t
    \nrm{e^{(t-\sigma)\Delta}\nabla \cdot (
      \ue(\cdot, \sigma)(|\nabla \ve(\cdot, \sigma)|^2+\ep)^{\frac{p-2}{2}}
      \nabla \ve(\cdot, \sigma))}{L^{\infty}}\,\intd{\sigma}
\label{uebound:2}
\end{align}
for all $t \in (0, \Tmaxe)$ and $\ep \in (0,1)$.
Since the condition \eqref{con:p-weak} implies that
$n(p-1) < \frac{n}{n\theta - 1}$ holds when $\frac{1}{n} < \theta \le 1$,
we can choose $q$ satisfying \eqref{vexLq:q} and $n(p-1) < q$,
and hence we can also pick $r > n$ fulfilling $r(p-1) < q$.
Now Lemma~\ref{Lem:vexLq} asserts that
%
\[
    \|\nabla \ve(\cdot, \sigma)\|_{L^q(\Omega)} \le c_1
    \quad\mbox{for all}\ \sigma \in (0, \Tmaxe)\ 
    \mbox{and}\ \ep \in (0,1).
\]
This together with the known smoothing properties of the Neumann
heat semigroup (\cite[Lemma~1.3~(iv)]{W-2010})
and the H\"{o}lder inequality enables us to see that
\begin{align}
\notag
    &\int_0^t
    \nrm{e^{(t-\sigma)\Delta}\nabla \cdot (
      \ue(\cdot, \sigma)(|\nabla \ve(\cdot, \sigma)|^2+\ep)^{\frac{p-2}{2}}
      \nabla \ve(\cdot, \sigma))}{L^{\infty}}\,\intd{\sigma}
\\ \notag
    &\quad\,\le c_2 \int_0^t (1+(t-\sigma)^{-\frac{1}{2}-\frac{n}{2r}})
    e^{-\lambda_1(t-\sigma)}
    \nrm{\ue(\cdot, \sigma)}{L^{\frac{qr}{q-r(p-1)}}}
    \nrm{\nabla \ve(\cdot, \sigma)}{L^q}^{p-1}\,\intd{\sigma}
\\
    &\quad\,\le c_1^{p-1} c_2 \int_0^t (1+(t-\sigma)^{-\frac{1}{2}-\frac{n}{2r}})
    e^{-\lambda_1(t-\sigma)}
    \nrm{\ue(\cdot, \sigma)}{L^{\frac{qr}{q-r(p-1)}}}\,\intd{\sigma}
\label{uebound:1}
\end{align}
for all $t \in (0, \Tmaxe)$ and $\ep \in (0,1)$,
where $\lambda_1 > 0$ is the first nonzero
eigenvalue of $-\Delta$ in $\Omega$ under the Neumann boundary condition.
Moreover, it follows from \eqref{LocalApp:mass} that
\[
    \nrm{\ue(\cdot, \sigma)}{L^{\frac{qr}{q-r(p-1)}}}
    \le \nrm{\ue(\cdot, \sigma)}{L^{\infty}}^{\kappa}
    \nrm{\ue(\cdot, \sigma)}{L^1}^{1-\kappa}
    = m^{1-\kappa} \nrm{\ue(\cdot, \sigma)}{L^{\infty}}^{\kappa}
\]
for all $\sigma \in (0, \Tmaxe)$ and $\ep \in (0,1)$, where
$\kappa := 1 - \frac{q - r(p-1)}{qr} \in (0,1)$.
In conjunction with \eqref{uebound:1}, this implies that
for any $T \in (0, \Tmaxe)$ and $\ep \in (0, 1)$,
\begin{align*}
    &\int_0^t
    \nrm{e^{(t-\sigma)\Delta}\nabla \cdot (
      \ue(\cdot, \sigma)(|\nabla \ve(\cdot, \sigma)|^2+\ep)^{\frac{p-2}{2}}
      \nabla \ve(\cdot, \sigma))}{L^{\infty}}\,\intd{\sigma}
\\
    &\quad\, \le c_3
    \left(\sup_{t \in (0,T)} \norm{\ue}{L^{\infty}}\right)^{\kappa}
\end{align*}
for all $t \in (0, T)$, because
$\int_0^{\infty} (1+\sigma^{-\frac{1}{2}-\frac{n}{2r}})
    e^{-\lambda_1\sigma}
    \,\intd{\sigma} < \infty$.
Inserting this into \eqref{uebound:2} and applying the maximum principle,
we thus find that
\[
    \sup_{t \in (0,T)}\norm{\ue}{L^{\infty}}
    \le c_4 + c_4 \left(\sup_{t \in (0,T)} \norm{\ue}{L^{\infty}}\right)^{\kappa}
\]
for all $T \in (0, \Tmaxe)$ and $\ep \in (0,1)$,
and thereby we conclude that \eqref{uebound:ue} and
\eqref{uebound:Tmaxe} hold.
\end{proof}

\section{Proof of Theorem~\ref{thm1}}\label{Sec:thm1}
In this section we will prove Theorem~\ref{thm1}.
We first give an $L^r$-estimate for $\ue$ in terms of the mass $m$,
which will be used to derive an $L^{\infty}$-estimate for $\ue - \baru$.

%
%
\begin{lem}\label{Lem:ueLr}
Suppose that $\theta \in (0,1]$, and that
$p$ satisfies \eqref{con:p-weak}.
Let $r \in (1, \infty)$.
Then there exists a constant $C > 0$ such that
\[
        \| \ue(\cdot, t) \|_{L^r(\Omega)}
        \le C e^{ -\frac{t}{r} }
                + Cm \left( 1 + m^{ \frac{2\theta (p-1)}{r} } + m^{ \frac{2\theta (p-1)}{r(1-a)} } \right)
\]
for all $t > 0$ and $\ep \in (0,1)$,
with some $a \in (0,1)$.
\end{lem}
\begin{proof}
We test the first equation of \eqref{KSPP-R}
by $\ue^{r-1}$ to obtain
  \begin{align}\label{ue-test}
\notag
        \frac{1}{r} \frac{\rm d}{\intd{t}} \int_{\Omega} \ue^r
         + \frac{4(r-1)}{r^2} 
             \int_{\Omega} 
               \big| \nabla \ue^{ \frac{r}{2} } \big|^2
        &= (r-1)\chi \int_{\Omega} 
           (|\nabla \ve|^2 + \ep)^{ \frac{p-2}{2} }
           (\nabla \ve \cdot \nabla \ue)
           \ue^{r-1}
\\ \notag
        &\leq (r-1)\chi \int_{\Omega} 
           (|\nabla \ve| + 1)^{p-1} |\nabla \ue| \ue^{r-1}
\\
        &= \frac{2(r-1)\chi}{r} \int_{\Omega}
           (|\nabla \ve| + 1)^{p-1}
           \big| \nabla \ue^{ \frac{r}{2} } \big| 
            \ue^{ \frac{r}{2} } 
  \end{align}
in $(0,\infty)$ for all $\ep \in (0,1)$.
On account of the Young inequality,
we can estimate
  \begin{align}\label{young-Lr}
    \frac{2(r-1)\chi}{r} \int_{\Omega}
        (|\nabla \ve| + 1)^{p-1}
        \big| \nabla \ue^{ \frac{r}{2} } \big| 
         \ue^{ \frac{r}{2} } 
    \le \frac{r-1}{r^2} \int_{\Omega} 
          \big| \nabla \ue^{ \frac{r}{2} } \big|^2
      + (r-1) \chi^2 \int_{\Omega}
          (|\nabla \ve| + 1)^{2(p-1)} \ue^r
  \end{align}
in $(0,\infty)$ for any $\ep \in (0,1)$.
Besides,
by the H\"{o}lder inequality,
we infer that
  \begin{align}\label{Holder-Lr}
      \int_{\Omega} (|\nabla \ve|+1)^{2(p-1)} \ue^r
      \le
        \left[
            \int_{\Omega} (|\nabla \ve| + 1)^s
        \right]^{ 
          \frac{2(p-1)}{s} 
          }
        \left[
            \int_{\Omega} \ue^{
                      \frac{sr}{s-2(p-1)}
                      }
        \right]^{
          \frac{s-2(p-1)}{s}
          }
  \end{align}
in $(0,\infty)$ for all $\ep \in (0,1)$, where
  \begin{align}\label{con:s-1}
    s \in (2(p-1), \infty).
  \end{align}
From \eqref{ue-test}, \eqref{young-Lr}
and \eqref{Holder-Lr} we have
  \begin{align*}
\notag
        \frac{1}{r} \frac{\rm d}{\intd{t}} \int_{\Omega} \ue^r
         + \frac{3(r-1)}{r^2} 
             \int_{\Omega} 
               \big| \nabla \ue^{ \frac{r}{2} } \big|^2
        \leq (r-1) \chi^2
        \left[
            \int_{\Omega} (|\nabla \ve|+1)^s
        \right]^{ 
          \frac{2(p-1)}{s} 
          }
        \left[
            \int_{\Omega} \ue^{
                      \frac{sr}{s-2(p-1)}
                      }
        \right]^{
          \frac{s-2(p-1)}{s}
          }
  \end{align*}
and hence,
  \begin{align}\label{ineq:ue-Lr}
\notag
        &\frac{\rm d}{\intd{t}} \int_{\Omega} \ue^r
         + \int_{\Omega} \ue^r
         + \frac{3(r-1)}{r} 
             \int_{\Omega} 
               \big| \nabla \ue^{ \frac{r}{2} } \big|^2
\\
        &\quad\, \leq r (r-1) \chi^2
        \left[
            \int_{\Omega} (|\nabla \ve|+1)^s
        \right]^{ 
          \frac{2(p-1)}{s} 
          }
        \left[
            \int_{\Omega} \ue^{
                      \frac{sr}{s-2(p-1)}
                      }
        \right]^{
          \frac{s-2(p-1)}{s}
          }
         + \int_{\Omega} \ue^r 
  \end{align}
in $(0, \infty)$ for all $\ep \in (0,1)$,
with $s$ satisfying \eqref{con:s-1}.
We now estimate the first term
on the right-hand side of \eqref{ineq:ue-Lr}.
According to Lemma~\ref{Lem:vexLq},
we know that
  \begin{align}\label{app:Lq-ve}
\notag
        \left[ \int_{\Omega} (|\nabla \ve|+1)^s \right]^{ 
            \frac{2(p-1)}{s} }
        &\le c_1(\| \nabla\ve(\cdot, t) \|_{
            L^s(\Omega)}^{2(p-1)} + 1)
\\ \notag
        &\le c_2 (
           e^{- (1 + \lambda_1) t} + m^{\theta}
           )^{2(p-1)} + c_1
\\
        &\le c_3 ( 1 + m^{2\theta(p-1)} )
  \end{align}
in $(0, \infty)$ for any $\ep \in (0,1)$,
where $s$ satisfies \eqref{con:s-1} as well as
  \begin{align}\label{con:s-2}
    s \in [1, \infty) \quad\mbox{if}\ 0 < \theta \le \frac{1}{n},
    \quad\mbox{and}\quad 
    s \in \left[1, \frac{n}{n\theta-1}\right) \quad\mbox{if}\ \frac{1}{n} < \theta \le 1.
  \end{align}
Moreover,
an application of the Gagliardo--Nirenberg inequality
(cf.\ \cite[Proposition~A.1]{DGS-2001})
shows that
  \begin{align}\label{app:GN-ue}
\notag
        \left[
            \int_{\Omega} \ue^{
                      \frac{sr}{s-2(p-1)}
                      }
        \right]^{
          \frac{s-2(p-1)}{s}
          }  
        &= \| \ue^{ \frac{r}{2} } \|_{
           L^{ \frac{2s}{s-2(p-1)} }(\Omega)
           }^2
\\ \notag
        &\le c_4 \| \nabla \ue^{ \frac{r}{2} } \|_{
              L^2(\Omega)
           }^{2a}
           \| \ue^{ \frac{r}{2} } \|_{
              L^{ \frac{2}{r} }(\Omega)
           }^{2(1-a)}
         + c_4 \| \ue^{ \frac{r}{2} } \|_{
              L^{ \frac{2}{r} }(\Omega)
           }^2
\\
        &= c_4 \| \nabla \ue^{ \frac{r}{2} } \|_{
              L^2(\Omega)
           }^{2a}
           m^{r(1-a)}
         + c_4 m^r
  \end{align}
in $(0, \infty)$ for every $\ep \in (0,1)$,
where $s$ fulfills \eqref{con:s-1}, \eqref{con:s-2} and
  \begin{align}\label{con:s-3}
    s \in (n(p-1), \infty),
  \end{align}
and $a := \frac{
    \frac{r}{2} - \frac{s-2(p-1)}{2s}
}{
    \frac{r}{2} + \frac{1}{n} - \frac{1}{2}
} \in (0,1)$.
Here
we can choose $s$
satisfying \eqref{con:s-1}, \eqref{con:s-2} and
\eqref{con:s-3},
since
the condition \eqref{con:p-weak}
enables us to see that
%
%
$n(p-1) <  \frac {n}{n\theta-1}$
when $\frac{1}{n} < \theta \le 1$.
Now,
invoking \eqref{app:Lq-ve}, \eqref{app:GN-ue}
and the Young inequality,
we obtain
  \begin{align}\label{fir-term}
\notag
        &r (r-1) \chi^2
        \left[
            \int_{\Omega} (|\nabla \ve| + 1)^s
        \right]^{ 
          \frac{2(p-1)}{s} 
          }
        \left[
            \int_{\Omega} \ue^{
                      \frac{sr}{s-2(p-1)}
                      }
        \right]^{
          \frac{s-2(p-1)}{s}
          }
\\ \notag
        &\quad\, \le c_5 (1 + m^{2\theta(p-1)})
            \| \nabla \ue^{ \frac{r}{2} } \|_{L^2(\Omega)}^{2a}
            m^{r(1-a)}
         + c_5 (1 + m^{2\theta(p-1)}) m^r
\\ \notag
        &\quad\, \le \frac{2(r-1)}{r} \int_{\Omega}
            \big| \nabla \ue^{ \frac{r}{2} } \big|^2
         + c_6 (1 + m^{2\theta(p-1)})^{ \frac{1}{1-a} } m^r
         + c_5 (1 + m^{2\theta(p-1)}) m^r
\\
        &\quad\, \le \frac{2(r-1)}{r} \int_{\Omega}
            \big| \nabla \ue^{ \frac{r}{2} } \big|^2
         + c_7 m^r \big(
             1 + m^{2\theta(p-1)} + m^{ \frac{2\theta(p-1)}{1-a} }
             \big)
  \end{align}
in $(0, \infty)$
for all $\ep \in (0,1)$,
with $s$ satisfying
\eqref{con:s-1}, \eqref{con:s-2} and \eqref{con:s-3}.
Next,
again by using the Gagliardo--Nirenberg inequality and the Young inequality,
we have
  \begin{align}\label{sec-term}
\notag
        \int_{\Omega} \ue^r
        &= \| \ue^{\frac{r}{2}} \|_{L^2(\Omega)}^2
\\ \notag
        &\le c_8 \| \nabla \ue^{ \frac{r}{2} } \|_{
              L^2(\Omega)
           }^{2b}
           \| \ue^{ \frac{r}{2} } \|_{
              L^{ \frac{2}{r} }(\Omega)
           }^{2(1-b)}
         + c_8 \| \ue^{ \frac{r}{2} } \|_{
              L^{ \frac{2}{r} }(\Omega)
           }^2
\\ \notag
        &= c_8 \| \nabla \ue^{ \frac{r}{2} } \|_{
              L^2(\Omega)}^{2b}
           m^{r(1-b)} + c_8 m^r
\\
        &\le \frac{r-1}{r} \int_{\Omega} 
            \big| \nabla \ue^{ \frac{r}{2} } \big|^2
         + c_9 m^r
  \end{align}
in $(0, \infty)$ for all $\ep \in (0,1)$,
where $b := \frac{
    \frac{r}{2} - \frac{1}{2}
}{
    \frac{r}{2} + \frac{1}{n} - \frac{1}{2}
} \in (0,1)$.
Plugging \eqref{fir-term} and \eqref{sec-term}
into \eqref{ineq:ue-Lr},
we finally derive the differential inequality
  \begin{align*}
        \frac{\rm d}{\intd{t}} \int_{\Omega} \ue^r
         + \int_{\Omega} \ue^r
        \le c_{10} m^r \big(
             1 + m^{2\theta (p-1)} + m^{ \frac{2\theta (p-1)}{1-a} }
             \big)
  \end{align*}
in $(0, \infty)$ for any $\ep \in (0,1)$.
Integrating this over $(0, t)$
for any $t \in (0, \infty)$,
we have
  \begin{align*}
      \int_{\Omega} \ue^r(\cdot, t)
      &\le e^{-t} \!\int_{\Omega} u_0^r
        + c_{10} m^r \big(
             1 + m^{2\theta (p-1)} + m^{ \frac{2\theta (p-1)}{1-a} }
             \big)
\\
      &\le c_{11} e^{-t}
        + c_{10} m^r \big(
             1 + m^{2\theta (p-1)} + m^{ \frac{2\theta (p-1)}{1-a} }
             \big)
  \end{align*}
for all $\ep \in (0,1)$,
which leads to the conclusion.
\end{proof}

We now derive
an $L^{\infty}$-estimate for $\ue - \baru$,
which is crucial to obtain \eqref{baru-m}.
In order to compute more directly, we introduce
\[
 \begin{cases}
    \uue(x,t) := \ue(x,t) - \baru,
  \\
    \vve(x,t) := \ve(x,t) - \baru^{\theta}
  \end{cases}
\]
for $\ep \in (0,1)$, $x \in \Ombar$ and $t > 0$.
Then $(\uue, \vve)$ satisfies the problem
\begin{align}\label{KSPP-R-2}
  \begin{cases}
    (\uue)_t=\Delta \uue - \chi \nabla \cdot 
                   \left(\ue (|\nabla \vve|^2 + \ep)^{\frac{p-2}{2}} \nabla \vve\right)
    & \mbox{in}\ \Omega \times (0, \infty),
  \\
    (\vve)_t=\Delta \vve - \vve + (\ue^{\theta} - \baru^{\theta})
    & \mbox{in}\ \Omega \times (0, \infty),
  \\
    \nabla \uue \cdot \nu=\nabla \vve \cdot \nu=0
    & \mbox{on}\ \pa\Omega \times (0, \infty),
  \\
    \uue(\cdot,0)=u_0 - \baru, \ \vve(\cdot,0)=v_0 - \baru^{\theta}
    & \mbox{in}\ \Omega.
  \end{cases}
\end{align}
%
%
%
\begin{lem}\label{Linf-uue}
Let $\theta \in (0,1]$, and suppose that $p$ satisfies \eqref{con:p-weak}.
Then there exist $C > 0$ and $t_1 > 0$ such that
\[
        \| \uue(\cdot, t) \|_{L^{\infty}(\Omega)}
        \le Cm^{\theta(p-1)+1} (1 + m^{\alpha} + m^{\beta})
    \quad
        \mbox{for all}\ t \ge t_1\ \mbox{and}\ \ep \in (0,1)
\]
with some $\alpha > 0$ and $\beta > 0$.
\end{lem}
\begin{proof}
%
%
We infer from the first equation of \eqref{KSPP-R-2} that
  \begin{align}\label{esti-uue}
\notag
        \| \uue(\cdot, t) \|_{L^{\infty}(\Omega)}
        &\le 
         \| e^{t\Delta} \uue(\cdot, 0) \|_{
             L^{\infty}(\Omega)}
\\ 
        &\quad\, 
         + \chi \int_{0}^{t}
          \| e^{(t - \sigma)\Delta} \nabla \cdot \big(
            \ue(\cdot, \sigma)
            ( | \nabla \vve(\cdot, \sigma) |^2 + \ep)^{
               \frac{p-2}{2}}
            \nabla \vve(\cdot, \sigma)
            \big) \|_{L^{\infty}(\Omega)}
        \, \intd{\sigma}
  \end{align}
for all $t > 0$ and $\ep \in (0,1)$.
We employ the Neumann heat semigroup estimate
(cf.\ \cite[Lemma~1.3~(i)]{W-2010})
to derive
  \begin{align}\label{esti-uue1}
\notag
         \| e^{t\Delta} \uue(\cdot, 0) \|_{
             L^{\infty}(\Omega)}
         &\le c_1 e^{-\lambda_1 t} 
                  \| \uue(\cdot, 0) \|_{L^{\infty}(\Omega)}
\\ 
         &\le c_2 e^{-\lambda_1 t}
  \end{align}
for all $t > 0$ and $\ep \in (0,1)$,
where $\lambda_1 > 0$ is the first nonzero
eigenvalue of $-\Delta$ in $\Omega$ under the Neumann boundary condition.
From now on, we estimate the second term
on the right-hand side of \eqref{esti-uue}.
We can apply \cite[Lemma~1.3~(iv)]{W-2010}
and the H\"{o}lder inequality to see that
  \begin{align}\label{uue-sec}
\notag  
      &\chi \int_{0}^{t}
          \| e^{(t - \sigma)\Delta} \nabla \cdot \big(
            \ue(\cdot, \sigma)
            ( | \nabla \vve(\cdot, \sigma) |^2 + \ep)^{
               \frac{p-2}{2}}
            \nabla \vve(\cdot, \sigma)
            \big) \|_{L^{\infty}(\Omega)}
        \, \intd{\sigma}
\\ 
      &\quad\, \le c_3 \int_{0}^{t} 
         \big( 
              1 + (t - \sigma)^{ 
                   -\frac{1}{2} - \frac{n}{2k}
                   }
            \big)
            e^{ - \lambda_1 (t - \sigma) }
              \| \nabla\vve(\cdot,\sigma) \|_{
                  L^{k_1 (p-1)}(\Omega)}^{p-1}
              \| \ue(\cdot,\sigma) \|_{
                  L^{k_2}(\Omega)}
                  \,\intd{\sigma}
  \end{align}
for all $t > 0$ and $\ep \in (0,1)$,
with $k_1 > n$ and $k_2 > n$ to be fixed later,
and $k > n$ satisfying
$\frac{1}{k} = \frac{1}{k_1} + \frac{1}{k_2}$.
In view of the obvious identity
$\nabla \vve = \nabla \ve$,
Lemma~\ref{Lem:vexLq} entails that
  \begin{align}\label{vve-k1}
     \| \nabla\vve(\cdot, \sigma) \|_{ L^{k_1 (p-1)}(\Omega)}^{p-1}
     \le c_4 (e^{-\lambda_1 (p-1)\sigma} + m^{\theta(p-1)})
  \end{align}
for all $\sigma > 0$ and $\ep \in (0,1)$,
where
  \begin{align}\label{con:k1-1}
    k_1 \in \left[\frac{1}{p-1}, \infty\right) \quad\mbox{if}\ 0 < \theta \le \frac{1}{n},
    \quad\mbox{and}\quad 
    k_1 \in \left[\frac{1}{p-1}, \frac{n}{(p-1)(n\theta-1)}\right) \quad\mbox{if}\ 
    \frac{1}{n} < \theta \le 1.
  \end{align}
We can actually choose $k_1 > n$ satisfying \eqref{con:k1-1},
since the condition \eqref{con:p-weak} ensures that
the relation
$n < \frac {n}{(p-1)(n-1)}$ holds
when $\frac{1}{n} < \theta \le 1$.
Recalling Lemma~\ref{Lem:ueLr},
we find a constant $a = a(k_2) \in (0,1)$ such that
  \[
    \| \ue(\cdot, \sigma) \|_{L^{k_2}(\Omega)}
    \le c_5 \!\left[ e^{ -\frac{\sigma}{k_2} }
           + m \left(
             1 + m^{ \frac{2\theta (p-1)}{k_2} } 
              + m^{ \frac{2\theta (p-1)}{ k_2(1-a)} } \right)
         \right]
  \]
for all $\sigma > 0$ and $\ep \in (0,1)$.
Collecting this and \eqref{vve-k1}
in \eqref{uue-sec}, we deduce that
  \begin{align}\label{I1to4}
\notag
    &\chi \int_{0}^{t}
          \| e^{(t - \sigma)\Delta} \nabla \cdot \big(
            \ue(\cdot, \sigma)
            ( | \nabla \vve(\cdot, \sigma) |^2 + \ep)^{
               \frac{p-2}{2}}
            \nabla \vve(\cdot, \sigma)
            \big) \|_{L^{\infty}(\Omega)}
        \, \intd{\sigma}
\\ \notag
    &\quad\, \le
       c_6 \int_{0}^{t} \big( 
              1 + (t - \sigma)^{ 
                   -\frac{1}{2} - \frac{n}{2k}
                   }
            \big)
            e^{ - \lambda_1 (t - \sigma) }
            e^{- \left(\lambda_1 (p-1) + \frac{1}{k_2}\right) 
            \sigma}
            \, \intd{\sigma}
\\ \notag
      &\qquad\,
        + c_6 m \left(
             1 + m^{ \frac{2\theta(p-1)}{k_2} } 
              + m^{ \frac{2\theta(p-1)}{ k_2(1-a)} } \right)
            \int_{0}^{t} \big( 
              1 + (t - \sigma)^{ 
                   -\frac{1}{2} - \frac{n}{2k}
                   }
            \big)
            e^{ - \lambda_1 (t - \sigma) }
            e^{ -\lambda_1 (p-1) \sigma}
            \, \intd{\sigma}
\\ \notag
      &\qquad\, 
        + c_6 m^{\theta(p-1)}
          \int_{0}^{t} \big( 
              1 + (t - \sigma)^{ 
                   -\frac{1}{2} - \frac{n}{2k}
                   }
            \big)
            e^{ - \lambda_1 (t - \sigma) }
            e^{ -\frac{\sigma}{k_2} }
            \, \intd{\sigma}
\\ \notag
      &\qquad\,
        + c_6 m^{\theta(p-1) + 1} \left(
             1 + m^{ \frac{2\theta(p-1)}{k_2} } 
              + m^{ \frac{2\theta(p-1)}{ k_2(1-a)} } \right)
            \int_{0}^{t} \big( 
              1 + (t - \sigma)^{ 
                   -\frac{1}{2} - \frac{n}{2k}
                   }
            \big)
            e^{ - \lambda_1 (t - \sigma) }
            \, \intd{\sigma}
\\ \notag
      &\quad\, =:
         c_6 I_1(\cdot, t) 
         + c_6 m \left(
             1 + m^{ \frac{2\theta(p-1)}{k_2} } 
              + m^{ \frac{2\theta(p-1)}{ k_2(1-a)} } \right)
            I_2(\cdot, t)
\\
      &\qquad\, \ 
        + c_6 m^{\theta(p-1)} I_3(\cdot, t)
        + c_6 m^{\theta(p-1)+1} \left(
             1 + m^{ \frac{2\theta(p-1)}{k_2} } 
              + m^{ \frac{2\theta(p-1)}{ k_2(1-a)} } \right)
           I_4(\cdot, t)  
  \end{align}
for all $t > 0$ and $\ep \in (0,1)$.
By virtue of \cite[Lemma~1.2]{W-2010}, we deduce
  \begin{align}
\notag  
    I_1(\cdot, t)
    &\le c_7 \big(1 + t^{
       \min \{ 0, 1 - \frac{1}{2} - \frac{n}{2k} \}} 
       \big)
           e^{-\!\min \!\big\{ \!
           \lambda_1, \lambda_1 (p-1) + \frac{1}{k_2}\!
            \big\} t}
\\
    &= 2c_7 e^{-\!\min \!\big\{ \!
           \lambda_1, \lambda_1 (p-1) + \frac{1}{k_2}\!
            \big\} t} \label{I1}
  \end{align}
for all $t > 0$, with $k_2$ satisfying
  \begin{align}\label{con:k2-1}
      \lambda_1 \neq \lambda_1 (p-1) + \frac {1}{k_2}.
  \end{align}
Also, \cite[Lemma~1.2]{W-2010} provides $\delta = \delta(p) \in (0,1]$
such that
  \begin{align}
\notag  
  I_2(\cdot, t)
    &\le c_8 \big(1 + t^{
       \min \{ 0, 1 - \frac{1}{2} - \frac{n}{2k} \}} 
       \big)
           e^{-\lambda_1 \delta t}
\\
    &= 2c_8 e^{-\lambda_1 \delta t} \label{I2}
  \end{align}
for all $t > 0$.
Similarly, we utilize \cite[Lemma~1.2]{W-2010} again to obtain
  \begin{align}
\notag  
    I_3(\cdot, t)
    &\le c_9 \big(1 + t^{
       \min \{ 0, 1 - \frac{1}{2} - \frac{n}{2k} \}} 
       \big)
           e^{-\!\min \!\big\{ \!
           \lambda_1, \frac{1}{k_2}\!
            \big\} t}
\\
    &= 2c_9 e^{-\!\min \!\big\{ \!
           \lambda_1, \frac{1}{k_2}\!
            \big\} t} \label{I3}
  \end{align}
for all $t > 0$, provided that
  \begin{align}\label{con:k2-2}
    \lambda_1 \neq \frac{1}{k_2}.
  \end{align}
On the other hand, recalling that $k > n$,
we see that
  \begin{align}\label{I4}
    I_4(\cdot, t) 
    = \int_{0}^{\infty} 
       (1 + \sigma^{-\frac{1}{2} - \frac{n}{2k}})
       e^{-\lambda_1 \sigma}
       \, \intd{\sigma}
    < \infty
  \end{align}
for all $t > 0$.
Now,
let $k_1 > n$ be as in \eqref{con:k1-1},
and take $k_2 > n$
large enough so that $k_2$ satisfies
\eqref{con:k2-1}, \eqref{con:k2-2} and
$\frac{1}{k_1} + \frac{1}{k_2} < \frac{1}{n}$.
Then, by plugging \eqref{I1}, \eqref{I2},
\eqref{I3} and \eqref{I4} into \eqref{I1to4}
we can derive that
  \begin{align*}
        &\chi \int_{0}^{t}
          \| e^{(t - \sigma)\Delta} \nabla \cdot \big(
            \ue(\cdot, \sigma)
            ( | \nabla \vve(\cdot, \sigma) |^2 + \ep)^{
               \frac{p-2}{2}}
            \nabla \vve(\cdot, \sigma)
            \big) \|_{L^{\infty}(\Omega)}
        \, \intd{\sigma}
\\ 
        &\quad\, \le c_{10} e^{-\!\min \!\big\{ \!
           \lambda_1, \lambda_1 (p-1) + \frac{1}{k_2}
            \big\} t}
          + c_{10} m \left(
             1 + m^{ \frac{2\theta(p-1)}{k_2} } 
              + m^{ \frac{2\theta(p-1)}{ k_2(1-a)} } \right)
                e^{-\lambda_1 \delta t}
\\
        &\qquad\,
          + c_{10} m^{\theta(p-1)} e^{-\!\min \!\big\{ \!
           \lambda_1, \frac{1}{k_2}\!
            \big\} t}
          + c_{10} m^{\theta(p-1)+1} \left(
             1 + m^{ \frac{2\theta(p-1)}{k_2} } 
              + m^{ \frac{2\theta(p-1)}{ k_2(1-a)} } \right)
  \end{align*}
for all $t > 0$ and $\ep \in (0,1)$,
with some $a=a(k_2)\in (0,1)$.
Inserting this and \eqref{esti-uue1} into
\eqref{esti-uue} yields
  \begin{align*}
    \| \uue(\cdot, t) \|_{L^{\infty}(\Omega)}
    & \le c_{2} e^{-\lambda_1 t}
          + c_{10} e^{-\!\min \!\big\{ \!
           \lambda_1, \lambda_1 (p-1) + \frac{1}{k_2}
            \big\} t}
\\
    &\quad\, + c_{10} m \left(
             1 + m^{ \frac{2\theta(p-1)}{k_2} } 
              + m^{ \frac{2\theta(p-1)}{ k_2(1-a)} } \right)
                e^{-\lambda_1 \delta t}
\\
        &\quad\,
          + c_{10} m^{\theta(p-1)} e^{-\!\min \!\big\{ \!
           \lambda_1, \frac{1}{k_2}\!
            \big\} t}
          + c_{10} m^{\theta(p-1)+1} \left(
             1 + m^{ \frac{2\theta(p-1)}{k_2} } 
              + m^{ \frac{2\theta(p-1)}{ k_2(1-a)} } \right)
  \end{align*}
for all $t > 0$ and $\ep \in (0,1)$,
and hence we arrive at the conclusion.
\end{proof}

We are now in a position to complete the proof of Theorem~\ref{thm1}.

\begin{prth1.1}
By virtue of \eqref{uebound:ue},
we argue as in the proof of \cite[Lemmas 4.2 and 4.3]{YL-2020}
to see that there exist
functions $u,v$ defined on $\Omega \times (0, \infty)$
as well as a sequence $(\ep_k)_{k\in\N} \subset (0,1)$ such that
%
\begin{align}
    &u_{\ep_k} \wsc u
    \quad\mbox{in}\ L^{\infty}(0, \infty; L^{\infty}(\Omega)),
    \label{final:u1}
\\ 
    &u_{\ep_k} \to u
    \quad\mbox{in}\ C^0_{\rm loc}([0, \infty); (W_0^{2,2}(\Omega))^*)
    \quad\mbox{and}
    \label{conv:dual}
\\
    &v_{\ep_k} \wsc v
    \quad\mbox{in}\ L^{\infty}(0, \infty; L^{\infty}(\Omega))
    \label{final:v1}
\end{align}
as $\ep = \ep_k \dwto 0$.
We first claim that $(u, v)$ is a global weak solution to \eqref{KSPP}.
To verify this we fix an arbitrary $T > 0$.
Then, due to arguments similar to those in \cite[Lemmas 3.2--3.5]{YL-2020},
we find a subsequence,
still denoted by $(\ep_k)_{k\in\N}$, and nonnegative functions $\ub, \vb$
such that
\begin{align}
\notag
    &u_{\ep_k} \to \ub
    \quad\mbox{a.e.\ in}\ \Omega\times (0, T),
\\
    &u_{\ep_k} \wc \ub
    \quad\mbox{in}\ L^2(\Omega \times (0,T)),
    \label{final:u2}
\\
    &v_{\ep_k} \to \vb
    \quad\mbox{in}\ L^{2}(0, T; W^{1,2}(\Omega)),
    \label{final:v2}
\\ \notag
    &\nabla v_{\ep_k} \to \nabla \vb
        \quad\mbox{a.e.\ in}\ \Omega\times (0, T)\quad\mbox{and}
\\ \notag
    &(|\nabla v_{\ep_k}|^2 + \ep_k)^{\frac{p-2}{2}}\nabla v_{\ep_k}
    \wc |\nabla \vb|^{p-2}\nabla \vb
    \quad \mbox{in}\ L^{p'}(\Omega \times (0, T))
\end{align}
as $\ep = \ep_k \dwto 0$,
and that $(\ub,\vb)$ is a weak solution of \eqref{KSPP} in $\Omega \times (0, T)$.
According to \eqref{final:u1}, \eqref{final:v1}, \eqref{final:u2} and \eqref{final:v2},
we observe that $(\ub, \vb) = (u,v)$.
In consequence, this would show that $(u,v)$ is a weak solution of \eqref{KSPP}
in $\Omega \times (0,T)$ for all $T > 0$,
and that hence the claim holds.

Now Lemma~\ref{Linf-uue} along with \eqref{final:u1} warrants that
there exist $t_1 > 0$ and a null set $N \subset [t_1, \infty)$ such that
\begin{align}\label{esti:without:null}
        \| u(\cdot, t) - \baru \|_{L^{\infty}(\Omega)}
        \le c_1 m^{\theta(p-1)+1} (1 + m^{\alpha} + m^{\beta})
    \quad
        \mbox{for all}\ t \in [t_1, \infty) \setminus N
\end{align}
for some $\alpha > 0$ and $\beta > 0$.
Indeed, from \eqref{final:u1} it follows that
\[
\ue - \baru \wsc u - \baru
    \quad\mbox{in}\ 
          L^{\infty}( \Omega \times [t_1, \infty) ) 
\]
as $\ep = \ep_k \searrow 0$,
and then due to the weak lower semicontinuity of the norm,
we infer from Lemma~\ref{Linf-uue} that
  \begin{align*}
    \|u - \baru\|_{L^{\infty}(\Omega \times [t_1, \infty))}
    &\le \liminf_{\ep = \ep_k \searrow 0}
    \|\ue - \baru\|_{L^{\infty}(\Omega \times [t_1, \infty))}
\\
    &\le c_1 m^{\theta(p-1)+1}(1 + m^{\alpha} + m^{\beta})
  \end{align*}
for some $\alpha > 0$ and $\beta > 0$,
and moreover the measure theory ensures
existence of a null set $N \subset [t_1, \infty)$ such that
\[
  u(\cdot, t) - \baru \in L^{\infty}(\Omega)
      \quad\mbox{and}\quad 
  \|u(\cdot, t) - \baru\|_{L^{\infty}(\Omega)}
  \le \|u - \baru\|_{L^{\infty}(\Omega \times [t_1, \infty))}
\]
for all $t \in [t_1, \infty) \setminus N$.
We claim that the inequality \eqref{esti:without:null} actually holds
for every $t \in [t_1, \infty)$.
In order to show this, first,
for each $t\in [t_1, \infty)$ we can find
$(\widetilde{t}_j)_{j\in\N} \subset [t_1, \infty) \setminus N$
such that
$\widetilde{t}_j \to t$ as $j \to \infty$,
and extracting a subsequence if necessary we also have
\[
  u(\cdot, \widetilde{t}_j) \wsc \widetilde{u}
    \quad\mbox{in}\ 
          L^{\infty}(\Omega)
          \ \mbox{as}\  j\to\infty
\]
with some $\widetilde{u} \in L^{\infty}(\Omega)$.
On the other hand, \eqref{conv:dual} implies
\[
  u(\cdot, \widetilde{t_j}) \to u(\cdot, t)
      \quad\mbox{in}\ 
          (W_{0}^{2,2}(\Omega))^{\ast}
          \ \mbox{as}\  j\to\infty.
\]
We thus have $\widetilde{u} = u(\cdot, t)$,
and due to the weak lower semicontinuity of the norm,
we arrive at
  \begin{align*}
      \| u(\cdot, t) - \baru \|_{L^{\infty}(\Omega)}
      & \le \liminf_{j\to\infty}
          \| u(\cdot, \widetilde{t}_j) - \baru \|_{
              L^{\infty}(\Omega)}
\\
      & \le c_1 m^{\theta(p-1)+1}(1 + m^{\alpha} + m^{\beta})
  \quad
    \mbox{for all}\ t \in [t_1, \infty),
  \end{align*}
which proves the claim,
and hence establishes \eqref{baru-m}.
For the latter part of the theorem,
let $\eta_0$ be such that
\[
  c_1 \eta_0^{\theta(p-1)} (1 + \eta_0^{\alpha} + \eta_0^{\beta}) \le 1,
\]
and for each $\eta \in (0, \eta_0)$
fix $m = \| u_0 \|_{L^1(\Omega)}$ such that $m \le \eta$.
Then we have
  \begin{align*}
    \| u(\cdot, t) - \baru \|_{L^{\infty}(\Omega)}
    & \le \eta \cdot 
      c_1 \eta_0^{\theta(p-1)}(1 + \eta_0^{\alpha} + \eta_0^{\beta})
\\
    & \le \eta
  \end{align*}
for all $t \ge t_1$, and the proof is complete.
\qed 
\end{prth1.1}

%
\smallskip
\section*{Acknowledgments}
The author would like to thank the referee, whose comments improved the manuscript.

%

\small

\end{document}